\documentclass[12pt]{amsart}
\usepackage{fullpage}
\usepackage{graphicx}

\usepackage{cases}
\usepackage{float}
\usepackage{hyperref}
\usepackage{blindtext}
\usepackage{multicol}
\usepackage{graphics}
\usepackage{graphicx}
\usepackage{epstopdf}
\usepackage{amssymb}
\usepackage{amsmath}
\usepackage{amsthm}
\usepackage{array}
\usepackage{latexsym}
\usepackage{amsfonts}
\usepackage{color}
\usepackage{verbatim}
\usepackage[numbers]{natbib}
\usepackage{amsthm,url,xspace}
\definecolor{couleurCitations}{rgb}{0,0,0.85}
\definecolor{couleurRef}{rgb}{0.75,0,0}

\newtheorem{theorem}{Theorem}[section]
\newtheorem{lemma}[theorem]{Lemma}
\newtheorem{proposition}[theorem]{Proposition}

\theoremstyle{definition}
\newtheorem{definition}[theorem]{Definition}

\title{Positive Artin Presentations}

\author{L. Armas-Sanabria,  M. Eudave-Mu\~noz, J.P.D\'{\i}az-Gonz\'{a}lez, G. Hinojosa-Palafox}

\address{Centro de Investigaci\'on en Ciencias, Instituto de Investigaci\'on en
Ciencias B\'asicas y Aplicadas, Universidad Aut\'noma del Estado de Morelos, Av. Universidad 1001, Cuernavaca, 62209, Morelos, M\'exico.}
\email{lorenaarmas089@gmail.com}

\address{Instituto de Matem\'aticas, Universidad Nacional Aut\'onoma de M\'exico, Campus Cuernavaca, Morelos, M\'exico}
\email{mario@matem.unam.mx}

\address{Centro de Investigaci\'on en Ciencias, Instituto de Investigaci\'on en
Ciencias B\'asicas y Aplicadas, Universidad Aut\'noma del Estado de Morelos, Av. Universidad 1001, Cuernavaca, 62209, Morelos, M\'exico.}
\email{juanpablo.diaz@uaem.mx}

\address{Centro de Investigaci\'on en Ciencias, Instituto de Investigaci\'on en
Ciencias B\'asicas y Aplicadas, Universidad Aut\'noma del Estado de Morelos, Av. Universidad 1001, Cuernavaca, 62209, Morelos, M\'exico.}
\email{gabriela@uaem.mx}

\thanks{The first author is supported by a fellowship {\em Investigadoras por M\'exico} from CONAHCYT. The second author is supported by PAPIIT UNAM grant IN117423.}

\date{January 2024}



\begin{document}

\begin{abstract}

An integral framed, closed pure $n$-braid $\hat \beta$ in $S^3$ describes  a positive Artin presentation, if $\beta$ can be put   on a disk with holes such that each relation describes  a  positive path  and these paths are disjoint.
In the present paper  we classify  the closed, pure $n$-braids $\hat \beta \subset S^3$,  such that $\beta$ represents  a positive Artin presentation. Also we prove that  if such $\beta$ describes a positive Artin presentation
then  $\hat \beta$ is strongly invertible.  
 Such  positive Artin presentation give us the fundamental group of closed, connected and orientable $3$-manifolds $M^3$, and in fact, by giving an example, we show that there exist $3$-manifolds whose fundamental group does not admit a positive Artin presentation. 
 \end{abstract}
 
 \maketitle

\section{Introduction}

In the $70$'s Gonz\'alez-Acu\~na following the work of Artin, defined what he called Artin $n$-presentations. He developed such a theory and showed that the fundamental group of  any closed, connected and orientable $3$-manifold has an Artin $n$-presentation. In fact, he  also proved that if a group given in terms of generators and relations has an Artin $n$-presentation then it is the fundamental group of a $3$-manifold \cite{GA}.
Furthermore, he showed that any closed, connected, orientable $3$-manifold can be seen as an open book with planar page and  disconnected binding and  that increasing the genus of the page it is possible to  obtain a connected binding.
 All this follows from the fact that any closed, connected and orientable $3$-manifold can be obtained by  Dehn surgery on an integral framed positive, closed pure $n$-braid,  (see \cite{L}, \cite{L1}, \cite{R}). Given an Artin $n$-presentation, this determines a 3-manifold $M$, in fact it determines a Heegaard diagram of a 3-manifold $M$ and a framed pure $n$-braid $\beta$ such that $M$ is obtained by Dehn surgery on $\hat \beta$, the closure of $\hat \beta$. Conversely, given a framed pure $n$-braid, this determines an Artin $n$-presentation.

In \cite{H}, Hempel shows that every    closed, connected and orientable 3-manifold  has a positive Heegaard diagram.
 Gonz\'alez-Acu\~na  motivated by the work of Hempel  and by the fact that the fundamental group of  any  such $3$-manifold has an Artin $n$-presentation, asked if it is possible that the fundamental group of any closed, connected and orientable $3$-manifold would have a positive Artin $n$-presentation (which is given by a special Heegaard diagram). 

Here we show that in general this is not true. So, there are $3$-manifolds whose fundamental group does not admit a positive Artin $n$-presentation. Also, we give a characterization of the closed, pure $n$-braids $\hat \beta$,  such that $\beta$ admits a positive Artin $n$-presentation. In the case of  closed, pure $3$-braids these are  links of the form
$\hat \beta = \widehat {\sigma_1^{2e_1}({\sigma_2\sigma_1\sigma_2})^{2e}}$ or $\hat \beta = \widehat {\sigma_2^{2f_1}({\sigma_2\sigma_1\sigma_2})^{2e}}$ (of course, depending of certain conditions on the integral framing and the integer numbers $e_1,f_1$ and $e$). 
The results of this paper are the following.

\vskip15pt

\begin{theorem} \label{mainintro} There is a set $\mathcal{P}$ of framed pure $n$-braids, which contains all the pure $n$-braids   which  admit a positive Artin n-presentation  (depending on the framings).
\end{theorem}

\vskip15pt

\begin{theorem} \label{stronglyinvertible} Let $\beta \in P_n$ be a pure $n$-braid and let $\hat \beta$ be the link obtained by closing $\beta$, contained in $S^3$. If $\beta $ produces   a positive Artin $n$-presentation, then $\hat \beta$ is strongly invertible.
\end{theorem}

The paper is organized as follows. In Section \ref{section2}, we define Artin presentations and survey some results in this theory which are used along the paper. In Section \ref{mainresults}, we describe a set
$\mathcal{P}$ of pure $n$-braids which produce positive
Artin $n$-presentations and prove the main results. In Section \ref{examples}, some examples are given of manifolds admitting a positive $n$-Artin presentation. It is also given an explicit example of a $3$-manifold whose fundamental group does not admit a positive Artin $n$-presentation.

\section {Preliminaries} \label{section2}
 
 In this Section, we define the concepts which will be used along the paper and give a survey of the main results in Artin Presentation Theory. For another exposition and developments of this theory see  \cite{Winkel}.
 
 \begin{definition} \label{artin}Given a presentation of a group $G = \langle x_1,x_2,\cdots,x_n : r_1,r_2,\cdots,r_n \rangle $, this presentation is an Artinian $n$-presentation if  it  satisfies the following equality $ \prod_{i=1}^{n} r_i x_i r_i^{-1} = \prod _{i=1}^n x_i$
in the free group $F_n  \hskip 10pt (: = F(x_1, x_2, \cdots , x_n) )$.
\end{definition}

The relations of the group in the  definition \ref{artin} can be seen geometrically. In what follows, let $D$ be a disk, let $V_1, V_2, \dots, V_n$ be a collection of disjoint disks in its interior, and for each $i$, let $y_i$ a point in the interior of $V_i$. Let $W=D-\textrm{int}(\cup V_i)$, so $W$ a disk with $n$ holes, whose boundary $\partial W$ consists of $n$ internal boundary components,  $\partial V_1, \partial V_2, \dots, \partial V_n$, and an external boundary component, denoted $\partial W_0 = \partial D$. Let $b_1,b_2, \dots,b_n$ be $n$ points cyclically ordered in the lower part of $\partial W_0$, as in Figure \ref{diskwithholes}. Let  $A_1, A_2,\cdots , A_n$  be disjoint arcs properly embedded in $W$, so that $\partial A_i$ consists of the point $b_i$ and a point $b_i'$ in $\partial V_i$. Let $A= \cup A_i$, note that $\textrm{int} (W-A)$ is a disk (see Figure \ref{diskwithholes}). If $p,q\in W-A$ are two points and $\gamma$ is any path from $p$ to $q$, transverse to $A$, then this $\gamma$ has associated  a word $w$ in the letters $x_1, x_2, \cdots , x_n$, where $x_i$ appears each time $\gamma$ crosses $A_i$ in a positive sense (that is, from left to right), and $x_i^{-1}$ appears each time $\gamma$ crosses $A_i$ in a negative sense (from right to left). We say that $\gamma$ represents $w$, which represents an element, also denoted by $w$, of $F(x_1, x_2, \cdots, x_n)$, the free group on $ x_1, x_2, \cdots, x_n$. 

Let $p_1,p_2,\dots,p_n$ be points cyclically ordered in the upper part of $\partial W_0$, in the interior of the arc determined by $b_1$ and $b_n$, as in Figure \ref{diskwithholes}. For each $i$, let $\gamma_i$ be a simple path going from $p_i$ to a point $a_i$ in $\partial V_i$, $a_i \not= b_i'$, i.e., $\gamma_i$ is an arc properly embedded in $W$, with one endpoint in $p$ and the other in $a_i$. Suppose  that the paths $\gamma_1,\gamma_2,\dots, \gamma_n$ are disjoint. Each $\gamma_i$ can be made transverse to $A$, and then determines a word $r_i$ in the free group  $F(x_1, x_2, \cdots, x_n)$. It is not difficult to see that the presentation $\langle x_1,x_2,\cdots,x_n : r_1,r_2,\cdots,r_n\rangle $ is an Artin $n$-presentation. In fact, the word  $ \prod_{i=1}^{n} r_i x_i r_i^{-1}$ can be represented by a simple path starting and ending in the point $p_1$, which is isotopic to a path going around $\partial W_0$, and which then represents the word $\prod _{i=1}^n x_i$.

\begin{figure}
\centering
\includegraphics[width=9 cm]{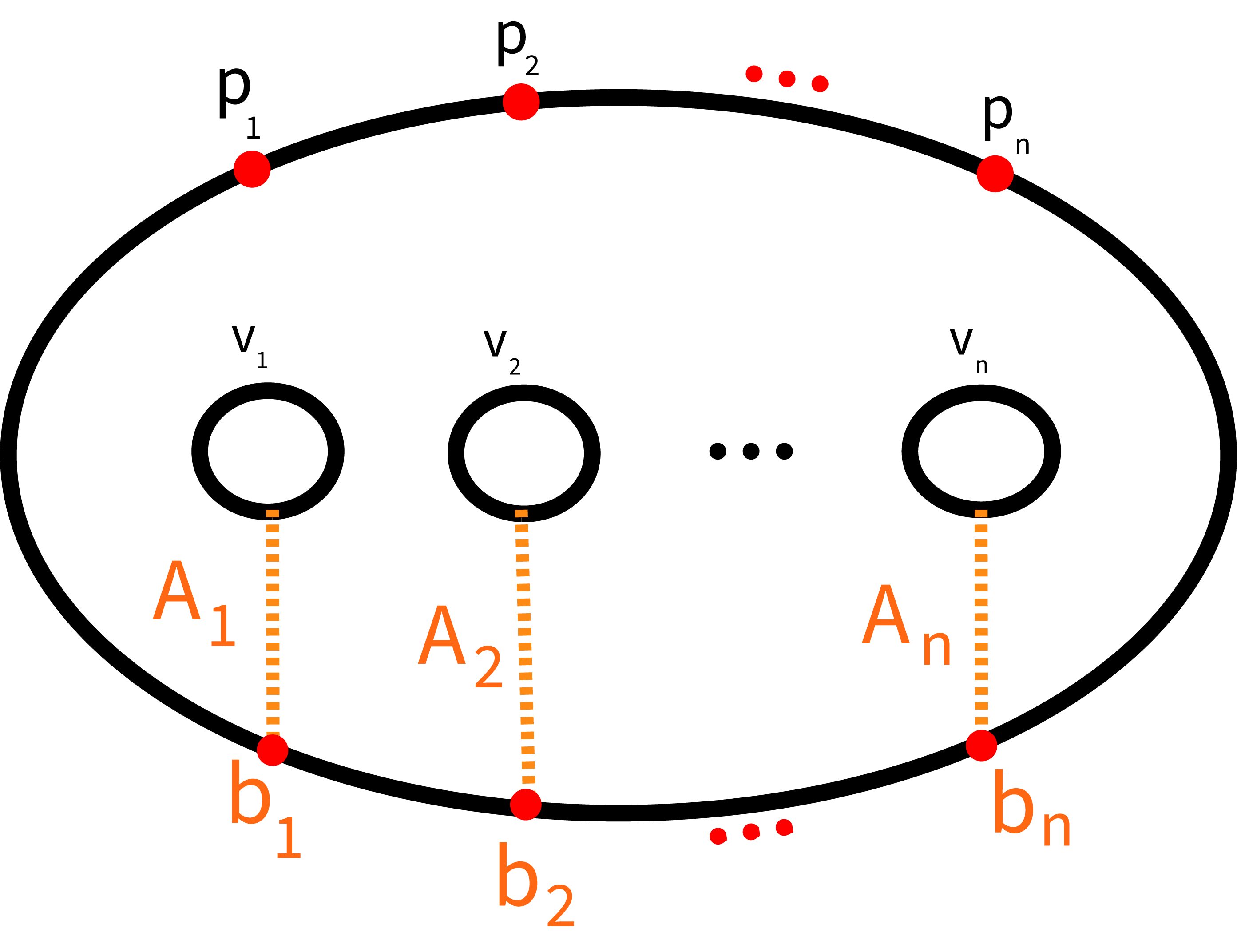}
\caption{The disk with $n$ holes }
\label{diskwithholes}
\end{figure}

Conversely, given an Artin $n$-presentation $\langle x_1,x_2,\cdots,x_n : r_1,r_2,\cdots,r_n \rangle $, it can be shown that there exists paths $\gamma_1,\gamma_2, \cdots, \gamma_n$ in $W$, with disjoint interiors, with $\gamma_i$ going from $p_i$ to a point in $\partial V_i$, whose associated word in the free group $F(x_1, x_2, \cdots, x_n)$ is precisely the word $r_i$. Examples of paths in a disk giving an Artin presentation are shown in Figures \ref{2diagram}, \ref{3diagram}, \ref{nonpositives}.

Other point of view that we can consider is the following. Consider the group of isotopy classes of homeomorphisms $h$ of $W$ such that $h\vert_{\partial W} =id_{\partial W}$. It is well known that this group can be identified with the integral framed pure $n$-braid group, and we denote it  by abuse of notation by $P_n$. Now we explain this relationship.

Consider the space $T=D\times I$, so we have a copy of the points $y_i$ in each of $D\times \{0 \}$ and $D\times \{1 \}$, denoted by $y_i^0$ and $y_i^1$, and also copies $V_i^0$ and $V_i^1$ of $V_i$. Let $\beta$ be a pure $n$-braid in $T$, i.e. it consists of $n$ disjoint arcs $\eta_1,\eta_2,\dots,\eta_n$, monotonic with respect to the product structure in $T$, and so that $\eta_i$ has endpoints in $y_i^0$ and $y_i^1$.
Let $N_i$ be a regular neighborhood of $\eta_i$, so that $N_i$ and $N_j$ are disjoint if $i\not= j$.
$N_i\cap D\times \{ 0 \}=V_{i}^0$, and similarly $N_i\cap D\times \{ 1 \} = V_{i}^1$. Choose points $a_{i}^{0}$, $a_{i}^{1}$ in each of $\partial V_{i}^0$ and $\partial V_{i}^1$, respectively (we can suppose that $a_{i}^1=a_{i}^0\times \{1\}$).

We define a framing on $\beta$ as a choice of $n$ arcs $\eta_i^*$, so that $\eta_i^*$ is contained in $\partial N_i$, has endpoints in  $a_i^0$ and $a_i^1$, but its interior is disjoint from the disks $V_i^0$ and $V_i^1$.
Note that $\eta_i$ and $\eta_i^*$ then determine a pure 2-braid in $T$, which is then equivalent to $\sigma_1^{2k_i}$, for some integer $k_i$. The collection of integers $k_1,k_2,\dots, k_n$ then determines the framing. We can do an isotopy on $T$, which fixes $\partial T - \textrm{int} D\times \{ 1 \}$, and which trivializes $\beta$, that is, at the end of isotopy we obtain a braid $\beta'$, which is a product braid $\{ y_1,y_2,\dots,y_n\}\times I$, and each arc $\eta_i^*$ becomes a product arc $a_i^0* \times I$. Then this defines a homeomorphism $h_\beta$ from $D\times \{ 1 \} -\textrm{int}(\cup V_i)$ to itself, which is the identity on $\partial (D\times \{ 1 \} -\textrm{int}(\cup V_i))$.
By identifying $W$ with $D\times \{ 1 \} -\textrm{int}(\cup V_i)$, we have shown that given $\beta$, we can construct a homeomorphism $h_\beta:W \rightarrow W$, which is the identity in all of $\partial W$.

Conversely, given a homeomorphism $h:W \rightarrow W$, such that $h\vert_{\partial W} =id_{\partial W}$,
we can find a pure framed $n$-braid $\beta$ on $D\times I$ whose associated homeomorphism is precisely $h$. To see this, take $T=D\times I$, and the trivial framed pure $n$-braid on it. $h$ can be seen as a homeomorphism of $D\times \{ 1 \}$, which is then isotopic to the identity homeomorphism, where $\partial D\times \{ 1 \}$ is being fixed during the isotopy. If this isotopy is realized in $T$, a pure $n$-braid is formed. 

Consider again a framed pure $n$-braid $\beta$ in $T$. Let $p_1^0,p_2^0,\dots,p_n^0$ be points cyclically ordered in the upper part of $\partial D \times\{ 0 \}$, take paths $B_1^0,B_2^0,\dots,B_n^0$ going from $p_i^0$ to $a_i^0$ for each $i$, assume that these paths are disjoint, and assume that are disjoint from the paths $A_i$'s defined above. Take the corresponding points $p_1^1,p_2^1,\dots,p_n^1$ in $\partial D \times\{ 1 \}$, and disjoint paths $B_1^1,B_2^1,\dots,B_n^1$ going from $p_i^1$ to $a_i^1$, which are disjoint from the $A_i^1$'s. Let $\eta_i^*$ be arcs that define a framing of $\beta$.
Let $\gamma_i^* = (B_i^1) \cup \eta_i^* \cup (A_i^0)^{-1} \cup (p_i^0 \times I)$, so this is a curve that starts and finishes at $p_i^1$. We consider it as a simple closed curve embedded in the boundary of the manifold $T -\textrm{int}(\cup N_i)$, which in fact is a genus $n$ handlebody. Note that considered in $T$, each curve $\gamma_i^*$ is a trivial curve, and the collection of curves $\gamma_i^*$ is just  a closure of the braid $\beta$. 

After performing the isotopy that trivializes $\beta$ note that  each $\gamma_i^*$ is transformed into a curve that consists of four arcs, $(B_i^0)^{-1}$ and $p_i^0 \times I$, which remain the same, $\eta_i^*$ which now looks like $a_i^0 \times I$, plus one arc $\gamma_i$ in $D\times \{ 1\}$ connecting $p_i^1$ to $a_i^1$. This shows that given a framed pure $n$-braid, it has associated  a collection of paths $\gamma_i$ in $W$, and then it determines an Artin $n$-presentation. Note that $\beta$ can be recovered from the $\gamma_i$'s, in fact, taking the collection of arcs $\eta_i^* \cup \gamma_i$, and pushing each of them into the interior of $D\times I$, keeping  its boundary fixed, we recover the braid $\beta$. Also, after performing the isotopy on $T$, the manifold $T -\textrm{int}(\cup N_i)$ can be identified with $W\times I$, and the curves $\gamma_i^*$ are simple closed curves on $\partial (W\times I)$, which determine a Heegaard diagram on $W\times I$ (see Figure \ref{Heegaard} for an example). 

\begin{figure}
\centering
\includegraphics[width=7 cm]{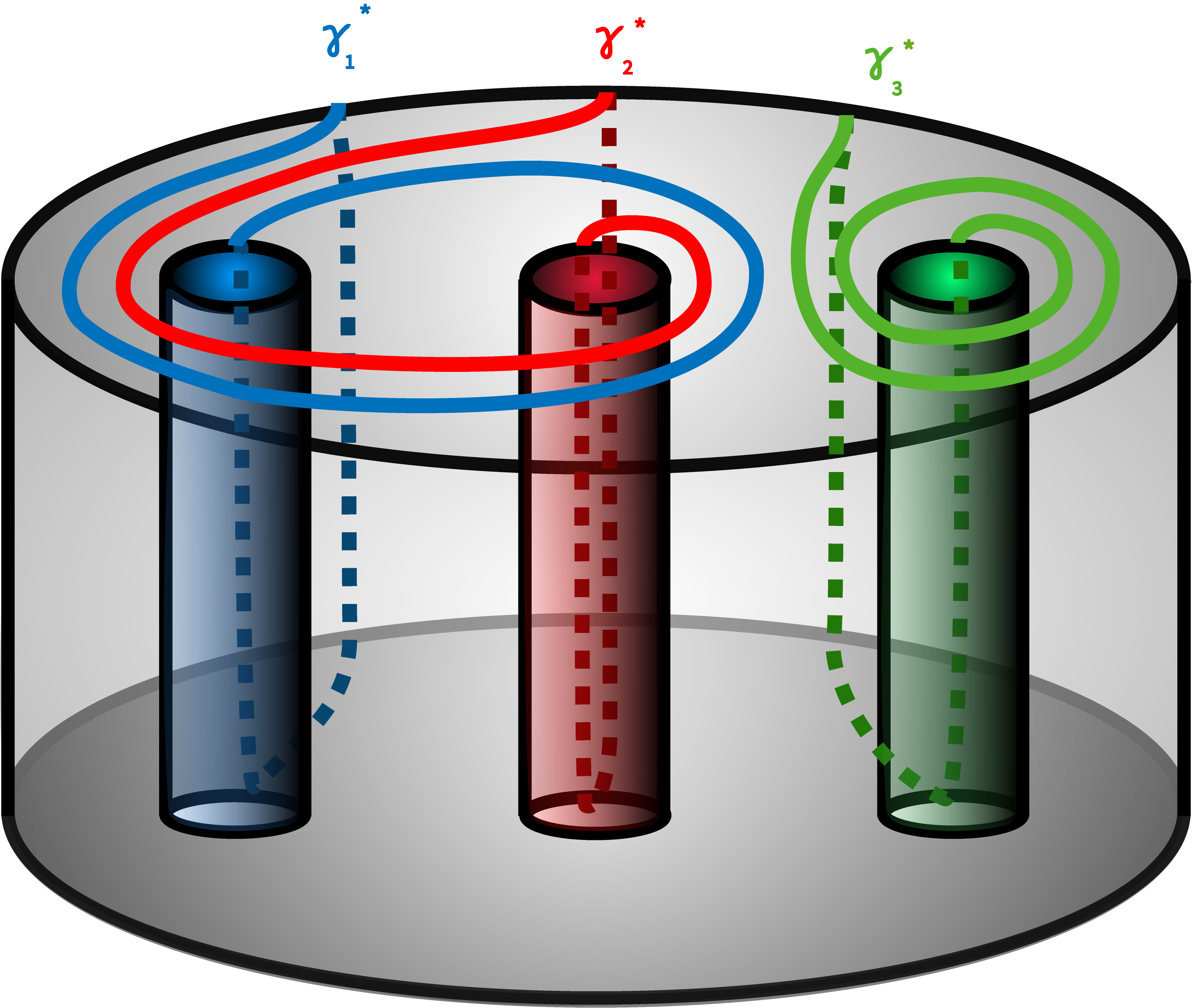}
\caption{A Heegaard diagram given by an Artin presentation}
\label{Heegaard}
\end{figure}

Consider the Heegaard diagram $(W\times I ; \gamma_1^*,\gamma_2^*,\dots, \gamma_n^*)$ just defined.  We can suppose that $W\times I$ is embedded in $S^3$ in an standard way. Let $M$ be the 3-manifold determined by this diagram, i.e. the manifold obtained by gluing 2-handles along $\gamma_i^*$ and then capping off with a 3-ball. Let $\alpha_i$ be the boundary of a regular neighborhood of $V_i^1\cup \gamma_i \cup p_i^1\times I \cup V_i^0$. This curve $\alpha_i$ is a trivial curve in the boundary of $W\times I\cup 2$-$handles$. So we can add 2-handles along all $\alpha_i$ and then along all $\gamma_i^*$ and then fill with many 3-balls, getting the same manifold $M$. Note that adding one 2-handle to each $\alpha_i$ we get a manifold with $n$ tori boundary components and one sphere boundary component. By capping off this sphere with a 3-ball, we get $S^3$ with the interior of $n$ tori removed. It follows that the resulting manifold is the exterior of the closed braid $\hat \beta$ defined above, and the curve $\gamma_i^*$ is just a framing on this link. 

Remember that if $L = \cup_1^n L_i$ is a rational framed link of $n$ components contained in $S^3$, then to do Dehn surgery on $L$  (following \cite{R}) consists in taking a regular neighborhood $N_i$ of $L_ i$  for all $i$, and a simple closed curve  $J_i$  in each $\partial N_i$, and then in this way, we get a new 3-manifold
 
  $$ M = S^3 -\cup \,\textrm{int} (N_i)\cup_h (T_1\cup \cdots \cup T_n)$$
 
  \noindent where  $T_i$ is a solid torus for all $i$, and $h$ is a union of homeomorphisms $h_i: \partial T_i \rightarrow \partial N_i \subset S^3$, such that a meridian curve $\mu _i$ of $T_i$ is glued along the curve $J_i$ determined by the framing, whose slope is  a rational number $p_i/q_i$ (with $p_i,q_i$ relative primes) or $\infty$. In this paper we are just considering integral framings.

This shows that if the Heegaard diagram $(W\times I ; \gamma_1^*,\gamma_2^*,\dots, \gamma_n^*)$ is considered to be in $S^3$, it determines a framed link so that doing Dehn surgery on this link we get the same 3-manifold $M$. In fact the curves $\gamma_i^*$ form a link in $S^3$, which is in fact the closure of the braid $\beta$, and the framing is given by the surface slope of the curves.

Note that the collection of disks $A_i\times I$ is a system of meridian disks for $W$, and each curve $\gamma_i^*$ intersects those disks only on $D\times \{1\}$, that is the intersection consists of the intersection of the arcs $\gamma_i$ with the arcs $A_i\times \{1 \}$. A presentation for $\pi_1(M)$ can be read from the Heegaard diagram $(W\times I ; \gamma_1^*,\gamma_2^*,\dots, \gamma_n^*)$, just by reading the intersection of the arcs $\gamma_i$ with the arcs $A_i\times \{1 \}$, and it follows that $\pi_1(M)= \langle x_1,x_2,\cdots,x_n : r_1,r_2,\cdots,r_n \rangle$, that is, $\pi_1(M)$  has an artinian presentation.

We have shown that given an artinian $n$-presentation $G= \langle x_1,x_2,\cdots,x_n : r_1,r_2,\cdots,r_n \rangle$, this determines a Heegaard diagram of a 3-manifold $M$ such that $\pi_1(M)=G$, and a closed pure framed $n$-braid $\beta$, such that $M$ is obtained by Dehn surgery on the closure $\hat \beta$ of $\beta$. Conversely, given a closed orientable 3-manifold $M$, by the Dehn surgery theorem of Lickorsih and Wallace (\cite{L}, \cite{L1}, \cite{R}), we know that, for some $n$, there is a closed pure $n$-braid $\hat \beta$ which produces $M$ by Dehn surgery on it. By considering the braid $\beta$ in a cylinder $D\times I$, and trivializing it by an isotopy, we can get $\pi_1(M)$ given by an artinian presentation, as explained above.   

We now show that given an Artin presentation, which determines a 3-manifold $M$, then this manifold  can be seen as an open book with planar page.
In general, given a compact, orientable surface $W$ with boundary, we can construct a $3$-manifold described as an open book with page $W$ in the following way \cite {GA}.
Let $h:  W  \rightarrow W $ be a homeomorphism which is the identity on $\partial W$.
In  $W \times [0,1]$ identify  $(x,1) $ to  $(h(x),0)$ for $x\in  W$ and $(x,t)$ to $ (x,0)$ for
$x\in \partial W, t\in [0,1]$. The resulting space is a closed  3-manifold $\tilde M = M(h)$ called an open book with monodromy $h$. If $H : W\times [0,1] \rightarrow M$ is the identification map then $H(\partial W \times \{ 0\}) $ is the binding and $H(W \times \{t\} ), t\in [0,1]$ is a page.

Let $W$ be a disk with $n$ holes, which as before is the disk $D$ with the interior of $n$ disks $V_1,V_2,\dots, V_n$ removed. Consider $W\times I$, and consider it embedded in $D\times I \subset S^3$. As before, let $a_i^0$ be a point in $\partial V_i^0$, for each $i$, let $p_1^0,p_2^0,\dots,p_n^0$ be points cyclically ordered in the upper part of $\partial D \times\{ 0 \}$, and let $B_1^0,B_2^0,\dots,B_n^0$ be paths going from $a_i^0$ to $p_i^0$ for each $i$.
Let $h:W \rightarrow W$ be a homeomorphism which is the identity on $\partial W$. Let $\gamma_i = h(B_i)$. Note that $h$ is completely determined by its image on the arcs $B_i$.  Consider the curves $\gamma_i^*$ on $W\times I$ defined by $\gamma_i ^* = \gamma_i \cup a_i^1\times I \cup (B_i^0)^{-1}\cup p_i^0\times I$. Do an isotopy on $W\times I$ fixing its boundary, except $W\times \{ 1 \}$, getting a framed pure $n$-braid in $D\times I$ (more precisely, the exterior of such a braid in $D\times I$). We identify $W\times \{ 0 \}$ with $W\times \{ 1 \}$, obtaining a space contained in $S^3$. Now the gluing homeomorphism is just the identity, so the punctured disks can be identified, getting a punctured solid torus in $S^3$, i.e., the exterior of a framed pure closed braid in a solid torus. Note that the exterior boundary of the solid tours is fibered by simple closed curves, which are preferred longitudes of the solid torus. Each one of this curves has to be identified to a point. This is equivalent to attaching a solid torus to $T$, such that each of the fibers bounds a meridian disk in the solid torus. But this is done by the complementary solid torus in $S^3$. So after the attaching solid torus, we get the exterior of a closed pure braid. Each of the holes of $W\times I$ is fibered by intervals parallel to $a_i^1\times I$, and after the identification, this is a torus fibered by curves parallel to a curve $a_i'$. Each of these curves is going to be identified to a point, which is equivalent to gluing a solid torus so that each of the fibers bound a meridian disk, i.e. to do then surgery with the given framing. This shows that the manifold $M$ has a open book decomposition with page a planar surface and binding consisting of $n+1$ simple closed curves.

So, we have shown the following results, first proved by Gonzalez-Acu\~na \cite{GA} in 1975.
\vskip 10pt

\begin{theorem}\label{fundamental}  The fundamental group of any orientable,  connected and closed  3-manifold has an Artin presentation. Furthermore, a group $G$ is the fundamental group of an 
orientable, connected and closed 3-manifold if it has an Artin npresentation.
\end{theorem}

\begin{theorem}\label{construction}
Given an Artin $n$-presentation $G = \langle x_1,x_2,\cdots,x_n : r_1,r_2,\cdots,r_n \rangle $, this determines a Heegaard diagram of a 3-manifold $M$, a framed link giving by Dehn surgery $M$ and an open book decomposition of $M$ with planar page.
\end{theorem}

\begin{definition} An artinian $n$-presentation  $\langle x_1,x_2\cdots,x_n : r_1,r_2,\cdots,r_n \rangle $  is positive if the word  describing $r_i$ in terms of the generators has only positive exponents for all $i = 1,2,3,\cdots, n$.
\end{definition}
\vskip 10pt

\begin{definition} Let $P_n$ be the group of integral framed pure $n$-braids and $\beta \in P_n$. The braid $\beta$ represents a positive artinian $n$-presentation if when isotoping it to a disk with $n$-holes, each chord represents a positive path $r_i$.
\end{definition}
\vskip 10pt

If $\beta$ represents a positive Artin presentation we say that $\hat \beta$, the closeure of $\beta$ describes a positive artinian $n$-presentation.

\vskip10pt

 \begin{definition} Let $L$ be a link of $n$-components in  $S^3$.  $L$ is {\it strongly invertible} if there exists an  involution $\tau$ of $S^3$, whose fixed point set consists of an unknotted circle $Fix(\tau)$, such that $\tau(L)=L$ and $Fix(\tau)$ intersects each component of $L$ in two points.
 \end{definition}

 \section{Main results}\label{mainresults}
\medskip
  
Consider the set $\mathcal{P}$  of pure   $n$-braids  $\beta $'s defined as follows. 
Let $P$ be a set of $n$ chords and   $\{ P_{k_1}, P_{k_2},\cdots , P_{k_l} \}$  a partition of $P$ such that
$|P_{k_j}| = n_j$,  whose elements are consecutive chords. $P_{k_1} $ has the consecutive chords $1,2,\cdots , n_1$, $P_{k_2} $
 has the chords $n_1+1,n_1+2,\cdots, n_1+n_2$ and so on;  and $n_1+ n_2+ \cdots + n_l = n$.
 
 Suppose that associated to $P_{k_j}$ is  $ \Delta_{n_j} ^{e_{s_j}} $ where 
 $\Delta_{n_j} ^{e_{s_j}}  $  consists of  $e_{s_j}$ positive full twists of $n_j$ consecutive chords 
 (we say that a twist along a closed curve is positive if it is counterclockwise), and  $e_{s_j}$ is a positive  integer for all ${s_j}$. 
 Furthermore, we can have unions of $P_{k_j} $'s consisting of consecutive elements of the partition
  in such way that if other block  of unions of $P_{k_j} $'s  contains   some  elements of a previous block then it contains all the previous block.

For example, $P _{k_j + k_{j+1}+k_{j+2}}  = P_{k_j}\cup P_{k_ {j+1}}\cup P_{k_{j+2}}$  
which   is   also    given $e_{s_j + s_{j+1}+s_{j+2}} $  positive  full twists, and we denote it  by  $\Delta _{n_j + n_{j+1}+n_{j+2}} ^{e_{s_j + s_{j+1}+s_{j+2}} } $.
 Take $\beta =  \prod P_{k_j}$ (consisting of elements of the partition and unions of consecutive subsets of the partition).

 We consider also the possibility that a block is given a negative twist but in a restricted manner. Suppose that we have
 $P _{k_j + k_{j+1}}= P_{k_j}\cup P_{k_{j+1}}$, which  is given at least one positive full twist $e_{s_j + s_{j+1}}$. Now take the block $P_{j+1}$ and perform exactly one negative full twist.
 
Now, let us see one example to illustrate the elements of $\mathcal {P}$:

\noindent Let $\beta \in \mathcal{P}$ be an element of $8$ chords and $P= \{ P_{k_1},P_{k_2}, P_{k_3} \} $ where 
$P_{k_1}= \{ 1,2,3\}, P_{k_2}= \{4,5\}$ and $P_{k_3} = \{6,7,8\}$ and also take$P_{k_1}\cup P_{k_2} = \{ 1, 2, 3, 4 ,5 \}, P_{k_1}\cup P_{k_2}\cup P_{k_3} = \{ 1, 2, 3, 4, 5, 6, 7, 8\}$.
Consider
$  \Delta _{n_1+n_2+n_3}^{e_{s_1+s_2+s_3} }, \Delta^{e_{s_1+s_2}}_{n_1+n_2}, \Delta _{n_1}^{e_{s_1}}, \Delta _{n_2}^{e_{s_2}}$ and $\Delta _{n_3}^{e_{s_3}}$. So take
$\beta = \Delta _{n_1+n_2+n_3}^{e_{s_1+s_2+s_3} }\cdot \Delta_{n_1+n_2}^{e_{s_1+s_2}}\cdot \Delta^{e_{s_3}}_{n_3}\cdot \Delta^{e_{s_2}}_{n_2}\cdot  \Delta^{e_{s_1}} _{n_1}$ (see Figure \ref{ex8}). Here we are assuming that all  $e_{s_1+s_2+s_3}, e_{s_1+s_2}, e_{s_1}, e_{s_2}, e_{s_3}$ are $\geq 0$. There is also the possibility that $e_{s_2}=-1$ (or $e_{s_3}=-1$), but in this case we have that $e_{s_1+s_2} \geq 1$ (or $e_{s_1+s_2+s_3} \geq 1$). But $e_{s_1+s_2+s_3}, e_{s_1+s_2}, e_{s_1}$ cannot be negative. 
 
\begin{figure}
\centering
\includegraphics[width=8cm]{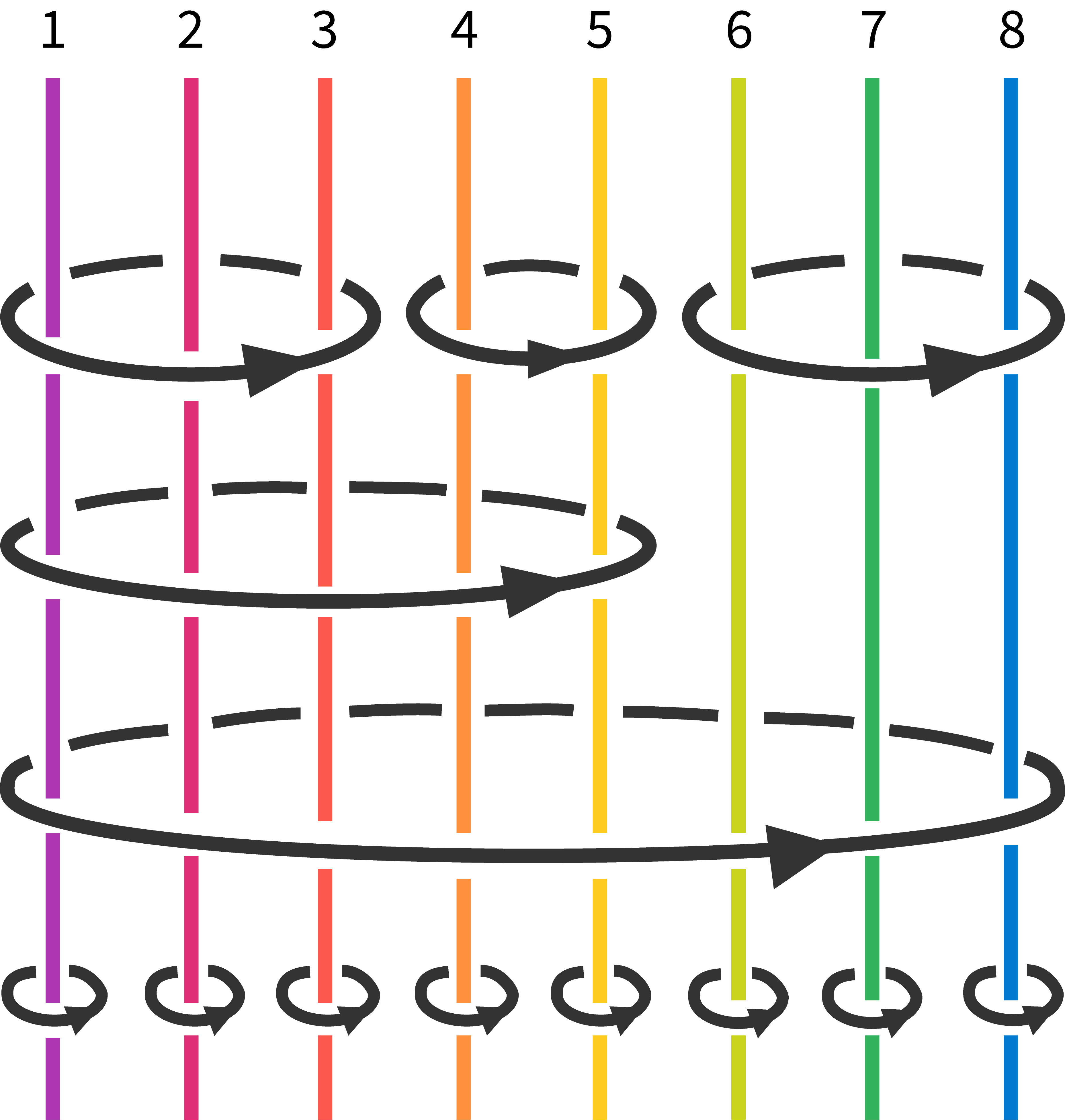} 
\caption{Example with 8 chords.}
\label{ex8}
\end{figure}

\begin{lemma}\label{strongly} Let $\beta \in \mathcal{P}$ be a pure $n$-braid and let $\hat \beta$ its closure. Then, when all chords of the braid are oriented from bottom to top, $\beta$ has a diagram where all crossings are negative. Furthermore $\hat \beta$ is a strongly invertible link. \end{lemma}

\begin{proof} When we take a partition of the braid and perform a full positive twist in a block, all resulting crossing points in the diagram will be negative, as seen in Figure \ref{Twist}. When we perform a negative twist some positive crossings will appear but this cancel with negative crossing produced in a previous twist. At the end there are only negative crossings.

Let $\beta_0$ be the trivial pure $n$-braid, so $\beta$ is obtained from $\beta_0$ by performing some twists.
Before performing each twist, we can take an unknotted curve enclosing the chords of the braid involved in the twist. Then the twist can be made by performing $1/e$-Dehn surgery on this curve. All such circles can be positioned in a plane $H$ with no intersections between them, where they are nested, such that $H$ is horizontal with respect to the braid and intersects each chord exactly once. Now consider the closed braid $\hat \beta_0$, it intersects the plane $H$ in $2n$ points, that is, each component intersects $H$ in two points. We can find a line $\bar u$ in $H$, which intersects each component of $\hat \beta_0$ in two points and also it intersects each of the circles in two points, this is possible because the circles are nested. Consider the link $\hat \beta_0 \cup L$, where $L$ is the collection of all circles. The line $\bar u$ intersects each component of the link $\hat \beta_0 \cup L$ in two points, and the link can be positioned so that it remains invariant after performing a $\pi$-rotation around $\bar u$. This shown that $\hat \beta_0 \cup L$ is a strongly invertible link. Now because $\hat \beta$ is obtained by Dehn surgery along the link $L$, it follows from \cite{M} that $\hat \beta$ is also strongly invertible.  
\end{proof}

\begin{figure}
\centering
\includegraphics[width=10cm]{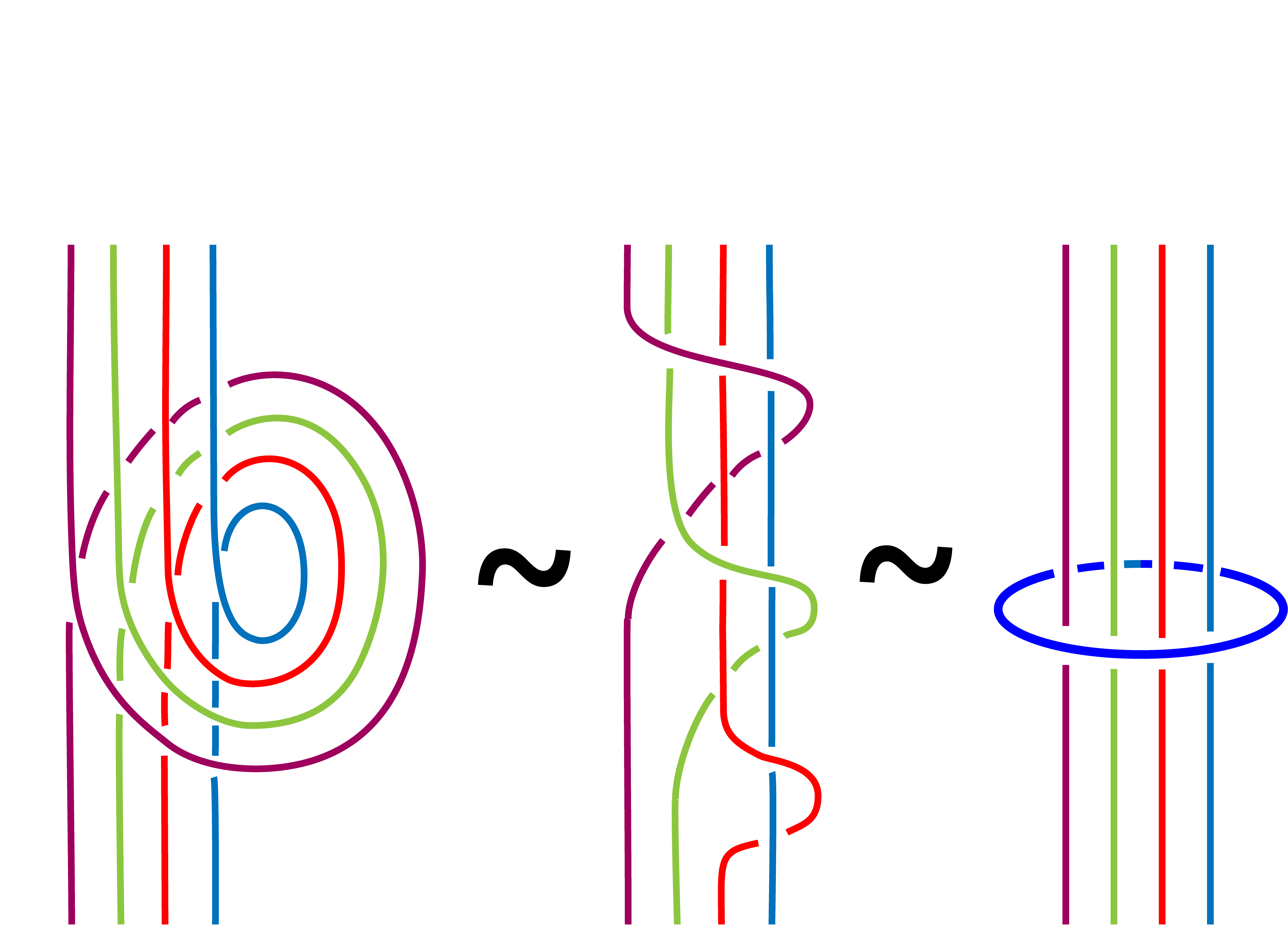}
\caption{The full twist along 4 chords}
\label{Twist}
\end{figure}

 By a graph manifold we mean, as usual, an orientable compact 3-manifold whose $JSJ$-decomposition consists of Seifert fibered manifold. By a degenerate graph manifold we mean a manifold which may contain degenerate Seifert fibered pieces, that is, pieces that may contain index zero exceptional fibers. These pieces are obtained, for example, when doing Dehn filling on a Seifert fibered manifold the slope of the filling coincides with a fiber of the Seifert fibration.
 
 \begin{lemma}\label{surgery} Let $\beta \in \mathcal{P}$ be a pure $n$-braid and let $\hat \beta$ its closure. Let $M$ be a 3-manifold obtained by Dehn surgery on all components of $\hat \beta$. Then $M$ is a (possibly degenerate) graph manifold. \end{lemma}

 \begin{proof} Let $\hat \beta_0 \cup L$ be the link defined in the proof of Lemma \ref{strongly}. Note that $\hat \beta_0 \cup L$ is contained in a solid torus $D\times S^1$, where $D$ is a disk contained in the plane $H$. For each component $\mu$ of $L$, $\mu \times S^1$ is a torus embedded in $D\times S^1$. These tori decompose the exterior of $\beta_0$ in Seifert fibered pieces; in fact each of the pieces is of the form $R\times S^1$, where $R$ is a planar surface contained in $H$,
 which may contain components of the braid, which will be of the form $*\times S^1$, where $*$ is a point in $D$, i.e. they are fibers in the Seifert fibration. The effect of doing Dehn surgery on the components of $L$ is just that of regluing the pieces by means of some homeomorphism, and then the exterior of $\beta$ is also a graph manifold. Now, because each component of $\beta$ is a fiber in one of the Seifert fibered pieces, by doing Dehn surgery on the components of $\hat \beta$, we still have a Seifert fibered piece, where some degenerate fiber may appear, depending on the surgery coefficients.
 \end{proof}

\begin{theorem}\label{main}
The set $\mathcal{P}$ of pure $n$-braids  $\beta$'s, contains all the pure $n$-braids, such that $\hat \beta $  admits a positive Artin presentation  (depending on the surgery coefficients).
 \end{theorem}

 \begin{proof}
Consider a disk $D$ with $n$ holes $V_1,\,V_2,\,\ldots,\,V_n$. As before, let $P$ be a set of $n$ disjoint paths $\gamma_{1}$, $\gamma_{2}$, $\cdots\,\gamma_{n}$ in  $W=D-\cup_{i=1}^n V_i$, each starting in a point $p_i$ in $\partial W_0$ and finishing in a point $a_i$ in $\partial V_i$. Suppose that $\{ P_{k_1}, P_{k_2},\cdots , P_{k_l} \}$ is a partition of $P$ such that $|P_{k_j}| = n_j$,  whose elements are consecutive paths. So $P_{k_1} $ has the consecutive paths $1,2,\cdots , n_1$, $P_{k_2} $ has the paths $n_1+1,n_1+2,\cdots, n_1+n_2$ and so on;  and $n_1+ n_2+ \cdots + n_l = n$. We are assuming that a path in $P_{k_i}$ determines a word in the symbols $x_{y+1},x_{y+2}, \cdots, x_{y+n_i}$, where $y=n_1 +n_2 + \cdots + n_{i-1}$. Note that any set of paths $P$ can be partitioned in such a manner, which includes the possibility that $P$ is just the same as $P_{k_1}$. Since $\beta$ is describing  a positive artinian diagram then if $\gamma_{k}$ intersects the generator $x_j$, then it must intersect it from left to right direction. Note that $\gamma_1$ intersects first the generator $x_1$, because if intersects first some other $x_i$, then it will no be possible for $\gamma_1$ to finish in $V_1$ without introducing negative intersections. We will describe all the possible diagrams  by induction on $n$.

\vskip 10pt
Suppose that $n=2$. Then we have a disk with 2 holes $V_1$ and $V_2$, and $P$ consists of two paths $\gamma_1$ and $\gamma_2$. Then there are two possible partitions for $P$.

\begin{figure}
\centering
\includegraphics[width=10cm]{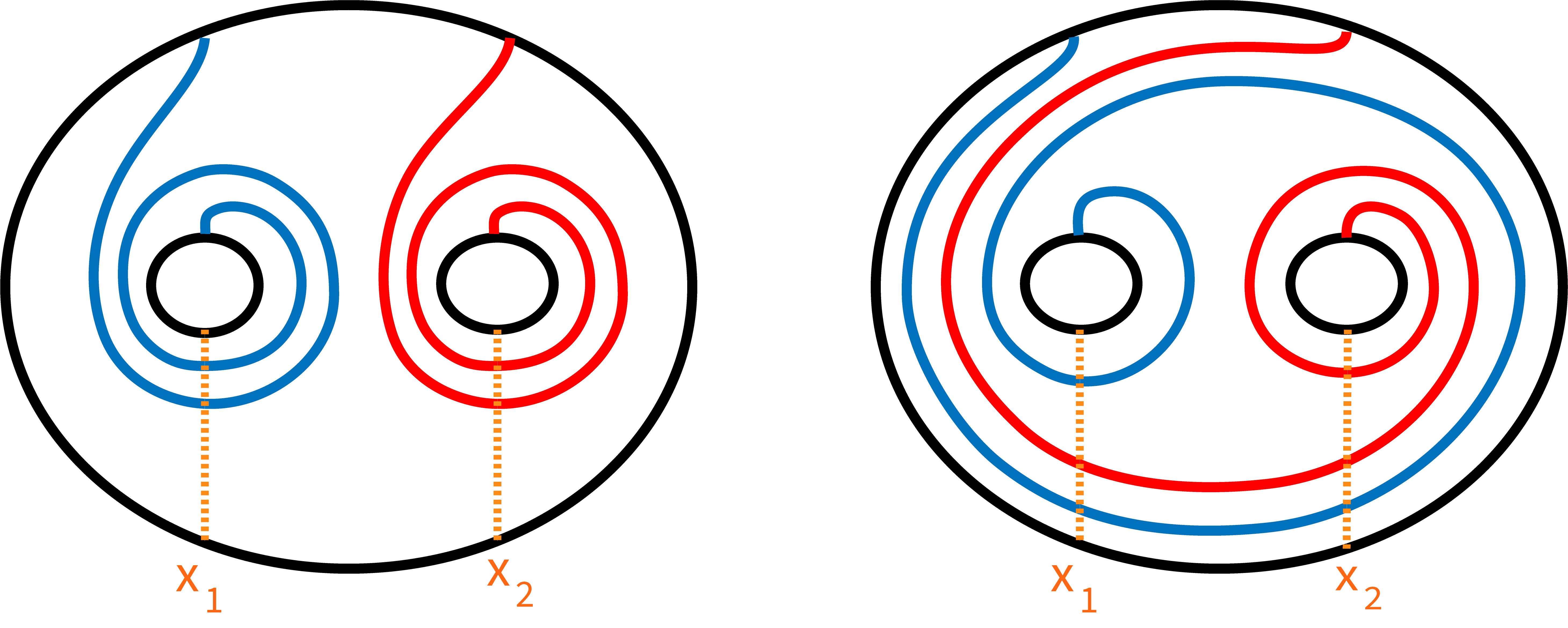} 
\caption{Diagrams of 2 holes.}
\label{2diagram}
\end{figure}

\begin{enumerate}
    \item The partition consists of the subsets $P_1=\{\gamma_1\}$ and $P_2=\{\gamma_2\}$. In this case $\gamma_1$ intersects only $x_1$, so $\gamma_1$ goes $m_1$ times around the hole $V_1$. In the same way, $\gamma_2$ intersects only $x_2$, hence $\gamma_2$ goes $m_2$ times around $V_2$, where $m_1,\,m_2\in\mathbb{N}$ (see Figure \ref{2diagram}).
    \item The partition of $P$ consists only of the set $P$.  In this case both $\gamma_1$ and $\gamma_2$ go $m$ times around both holes. Then $\gamma_1$ makes $m_1$ turns around the hole $V_1$ and $\gamma_2$ makes $m_2$ turns around the hole $V_2$, where $m,\,m_1,\,m_2\in\mathbb{N}$ (see Figure  \ref{2diagram})
\end{enumerate}
\vskip 10pt
We will also describe the case $n=3$. In this case, we have four possible partitions for $P$. We would like to point out, as we will see, that this is a recursive process.

\begin{figure}
\centering
\includegraphics[width=10cm]{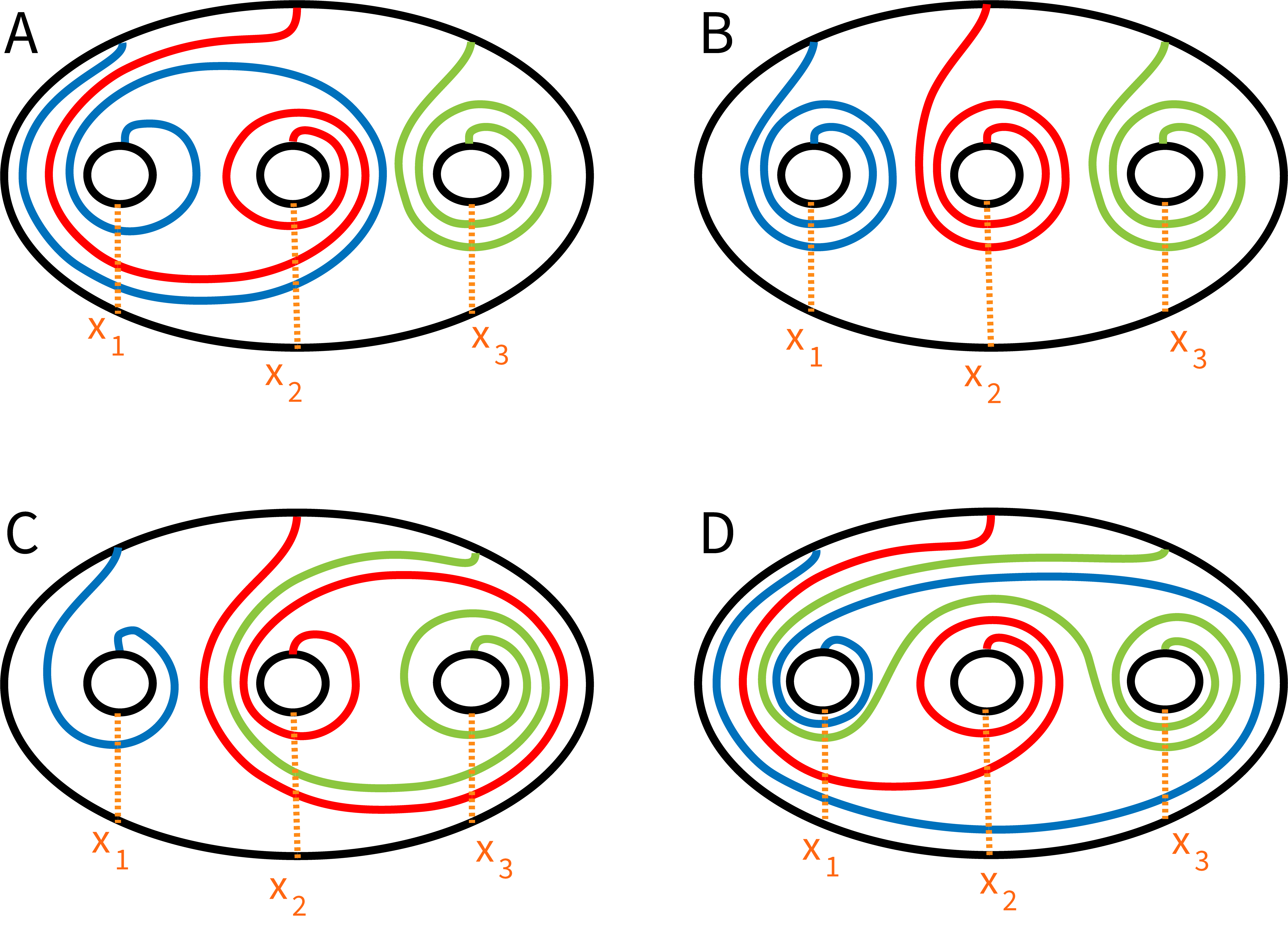} 
\caption{Diagrams of 3 holes.}
\label{3diagram}
\end{figure}

\begin{enumerate}
    \item The partition consists of the subsets $P_1=\{\gamma_1\}$, $P_2=\{\gamma_2\}$ and $P_3=\{\gamma_3\}$. Then, $\gamma_i$ ($i=1,2,3$), intersects only $x_i$ and goes $m_i$ times around the hole $V_i$, where $m_1,\,m_2,\,m_3\in\mathbb{N}$. See Figure \ref{3diagram} B.
    \item The partition consists of the subsets $P_1=\{\gamma_1\}$, 
    $P_{2,3}=\{\gamma_2,\,\gamma_3\}$. Then $\gamma_1$ goes $m_1$ times around $V_1$ and both $\gamma_2$ and $\gamma_3$ skip the hole $V_1$. Now we have two possibilities for $P_{2,3}=\{\gamma_2,\,\gamma_3\}$ described in the case $n=2$; \emph{i.e.}
    \begin{enumerate}
        \item There is a subpartition of it consisting of the subsets $P_2=\{\gamma_2\}$ and $P_3=\{\gamma_3\}$. In this case $\gamma_2$ goes $m_2$ times around the hole $V_2$. In the same way, $\gamma_3$ skips the hole $V_2$ and goes $m_3$ times around  $V_3$, where $m_1,\,m_2,\,m_3\in\mathbb{N}$. This is equivalent to Case 1.
    \item The set $P_{2,3}$.  In this case both $\gamma_2$ and $\gamma_3$ turn around both holes $k_2$ times. Then $\gamma_2$ goes $m_2$ times around the hole $V_2$ and $\gamma_3$ goes $m_3$ times around the hole $V_3$, where $m,\,m_1,\,m_2\in\mathbb{N}$ (see Figure \ref{3diagram} C).
    \end{enumerate}
    \item The partition consists of the subsets $P_{1,2}=\{\gamma_1,\,\gamma_2\}$, 
    $P_{3}=\{\gamma_3\}$. This case is analogous to the previous one, then $\gamma_3$ turns $m_3$ times around $V_3$ and both $\gamma_1$ and $\gamma_2$ avoid the hole $V_3$. Again, we have two possibilities for $P_{1,2}=\{\gamma_1,\,\gamma_2\}$ described in the case $n=2$; i.e.
    \begin{enumerate}
        \item There is a subpartition of it consisting of the subsets $P_1=\{\gamma_1\}$ and $P_2=\{\gamma_2\}$. In this case, $\gamma_1$ turns $m_1$ times around the hole $V_1$. In the same way, $\gamma_2$ avoids the hole $V_1$ and goes $m_2$ times around  $V_2$, where $m_1,\,m_2,\,m_3\in\mathbb{N}$. This is equivalent to Case 1.
    \item The set $P_{1,2}$.  In this case both $\gamma_1$ and $\gamma_2$ go around both holes $m$ times. Then $\gamma_1$  goes $m_1$ times around $V_1$ and $\gamma_2$ turns $m_2$ times around the hole $V_2$, where $m,\,m_1,\,m_2\in\mathbb{N}$ (see Figure \ref{3diagram} A).
    \end{enumerate}
    \item The partition consists only of the set $P$. Thus, the paths $\gamma_1$, $\gamma_2$ and $\gamma_3$ go all $k_3$ times around  the three holes, where $k_3\in\mathbb{N}$, i.e. the word $(x_1x_2x_3)^{k_3}$ appears in each of the chords. 
    There are several cases.
    
    \begin{enumerate} 
    \item First, $\gamma_1$ goes once around the holes $V_1$, $V_2$ and $V_3$ and finally turns $m_1$ times around the hole $V_1$. For the chords $\gamma_2$ and $\gamma_3$, we have again two possibilities as in the case $n=2$.
        \begin{enumerate}
        \item The paths $\gamma_2$ and $\gamma_3$ turn $k_2$ times around the holes $V_2$ and $V_3$. Then 
        $\gamma_2$  goes $m_2$ times around $V_2$  and $\gamma_3$ turns $m_3$ times around the hole $V_3$, where $k_2,\,m_1,\,m_2\in\mathbb{N}$ (This is like Figure \ref{3diagram} B with some full twists).
        \item The path $\gamma_2$ makes $m_2$ turns around the hole $V_2$. In the same way, $\gamma_3$ avoids the hole $V_2$ and goes $m_3$ times around  $V_3$, where $m_2,\,m_3\in\mathbb{N}$ (see Figure \ref{3diagram} D).
        \end{enumerate}
        \item Second, $\gamma_1$ goes once around the holes $V_1$, $V_2$ and $V_3$, then goes $k_2$ times around the holes $V_1$ and $V_2$, and finally turns $m_1$ times around the hole $V_1$. The chord $\gamma_2$ goes also goes $k_2$ times around $V_1$ and $V_2$ and then goes $m_2$ times around $V_2$. Finally, $\gamma_3$ goes $m_3$ times around $V_3$, where $k_2, m_1,m_2,\,m_3\in\mathbb{N}$. We have something like in Figure \ref{3diagram} C, but with full twists.
    \end{enumerate}

\end{enumerate}

Continuing inductively, we consider the set $P$ of $n$ paths,  where  $\{ P_{k_1}, P_{k_2},\cdots , P_{k_l} \}$ is a partition of $P$ such that $|P_{k_j}| = n_j$,  whose elements are consecutive paths. We will start our analysis, considering each subset $P_{k_j}$. As we mentioned above, $P_{k_j}$ is a subset of $n_j$ consecutive paths
$\gamma^{k_j}_{1},\,\gamma^{k_j}_{2},\,\ldots,\,\gamma^{k_j}_{n_j}$ in $W = D-\cup_{i=1}^n V_i$, starting in a point $p_i^{k_j}$ in $\partial W_0$ and finishing in a point in $V_i^{k_j}$. Consider the corresponding set of holes $H_{k_j}=\{V^{k_j}_{1},\,V^{k_j}_{2},\,\ldots,\,V^{k_j}_{n_j}\}$. Notice that if $\gamma^{k_j}_{1}$ skips the hole $V^{k_j}_s$, then the chord $\gamma^{k_j}_{s}$ must intersect the curve described by $\gamma^{k_j}_{1}$, contradicting that $\beta$ describes a positive Artin diagram. Hence, we can assume that $\gamma^{k_j}_{1}$ intersects each generator $x^{k_j}_i$ for $2\leq i\leq n_j$, otherwise we have a new partition for our set $P_{k_j}$. Summarizing $\gamma^{k_j}_{1}$ goes $h^{k_j}$ times around all the holes of $H_{k_j}$ and suppose that it goes $m^{k_j}_1$ times around the hole $V^{k_j}_{1}$, where $h^{k_j},\,m^{k_j}_1\in\mathbb{N}$. The case when $\gamma^{k_j}_{1}$ goes around a subset of the holes before going around $V_{k_j}$ can be done similarly.\\

\begin{figure}
\centering
\includegraphics[width=10cm]{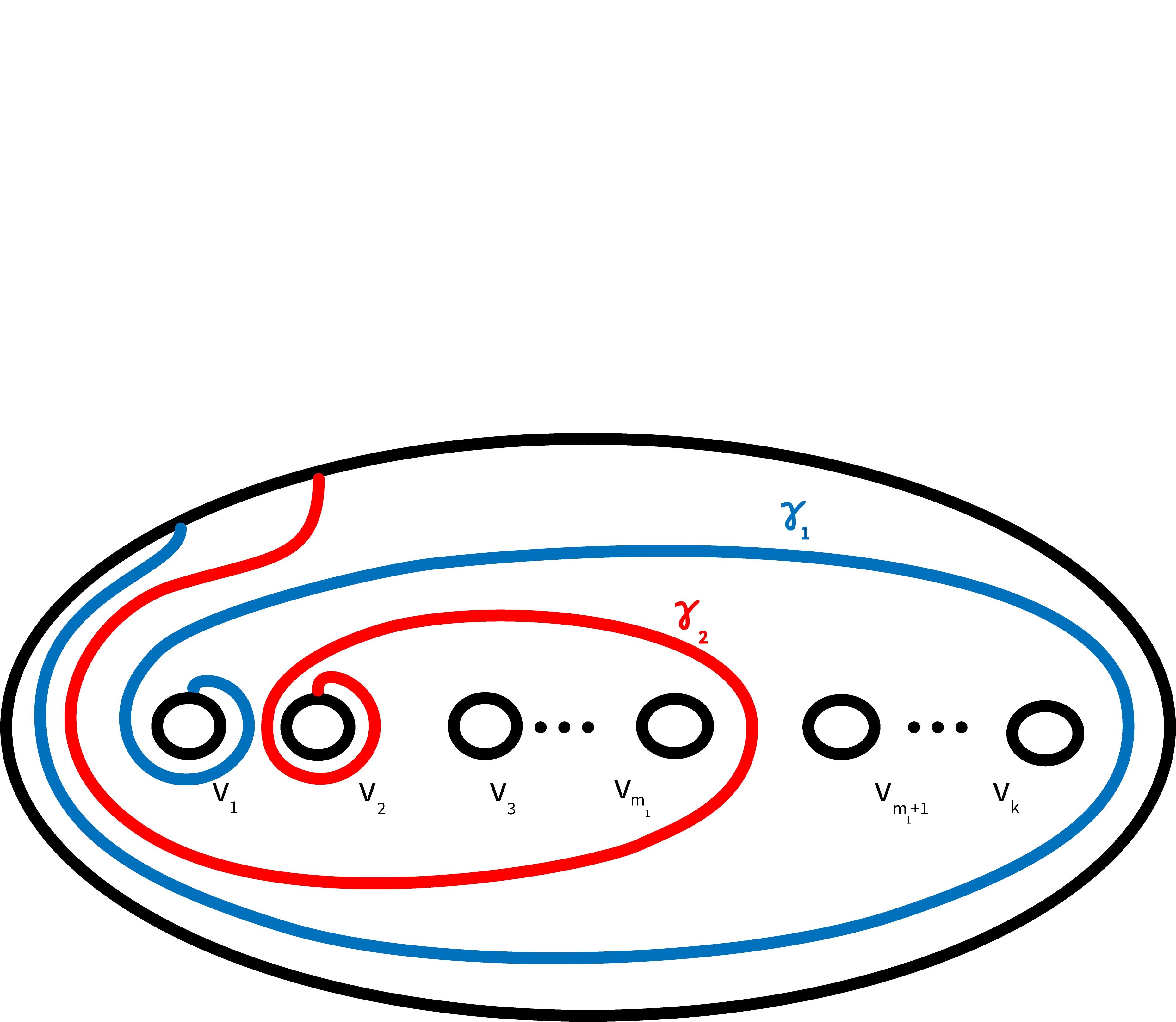} 
\caption{Diagram of $n$ holes.}
\label{ndiagram}
\end{figure}

\begin{figure}
\centering
\includegraphics[width=10cm]{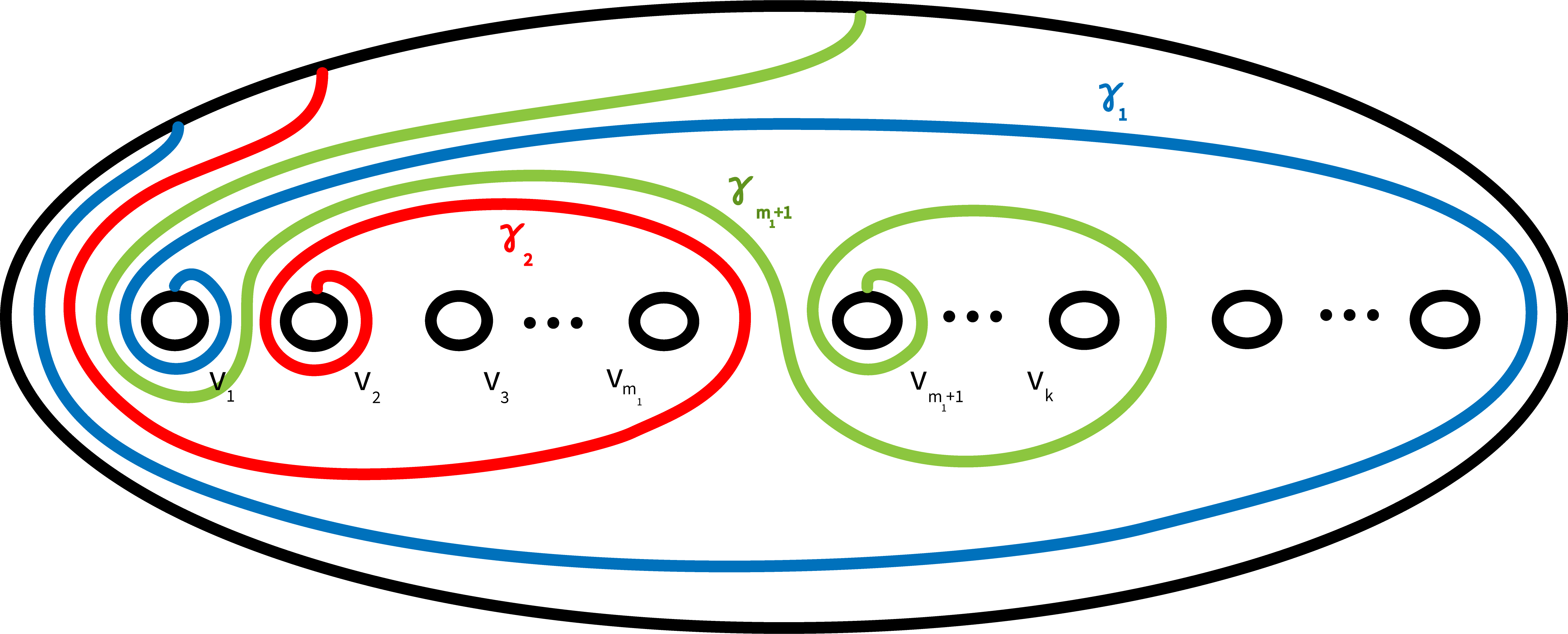} 
\caption{3 chords in a diagram of $n$ holes.}
\label{ndiagrams}
\end{figure}

Next, consider $\gamma^{k_j}_{2}$. First it goes around all the holes of $H_{k_j}$, $h^{k_j}-1$ times. Then again $\gamma^{k_j}_{2}$ goes around of a subset 
$H^{k_j}_2$ of holes. Observe that, by the same previous argument, $H^{k_j}_2$ consists of consecutive holes, $V^{k_j}_1,\,V^{k_j}_2,\,\ldots,\,V^{k_j}_{s_2}$. Thus, we have two possibilities.
\begin{enumerate}
    \item If $H^{k_j}_2$ is a proper subset of $H_{k_j}$, we have a subpartition of $H_{k_j}$ or equivalently a subpartition of $P_{k_j}$. In this case, $\gamma^{k_j}_{2}$ goes around all the holes of $H_{k_j}$, $h^{k_j}$ times, then turns $h^{k_j}_2$ times around all the holes of 
    $H^{k_j}_2$ and finally goes $m^{k_j}_2$ times around the hole $V^{k_j}_{2}$, where $h^{k_j}_2,\,m^{k_j}_2\in\mathbb{N}$.
    Now, we consider the path $\gamma^{k_j}_{3}$, and again, $\gamma^{k_j}_{3}$ can go around all the holes of $H^{k_j}_2$ or a proper subset of it. We continue this analysis in a recursive way, and since $H^{k_j}_2$ is finite, we finish at some finite time. 
    
    \item If $\gamma^{k_j}_{2}$ goes around all the holes of $H_{k_j}$, $h^{k_j}$ times, then goes $m^{k_j}_2$ times around the hole $V^{k_j}_{2}$.
\end{enumerate}

Notice that the path $\gamma^{k_j}_{s_2+1}$ must skip the subset 
    $H^{k_j}_2-\{V^{k_j}_1\}$ of holes and goes around of a subset 
    $H^{k_j}_3=\{V^{k_j}_{s_2+1},\,V^{k_j}_{s_2+2},\,\ldots,\,V^{k_j}_{s_3}\}$ of holes; otherwise $V^{k_j}_{s_2+1}$ will not be reached by $\gamma^{k_j}_{s_2+1}$. Next, we repeat the previous analysis and again, since the cardinality of $H^{k_j}_3$ is finite, we finish at some finite time.
 We continue describing each chord of $P_{k_j}$ $j=1,\ldots,l$ until we describe the last chord $\gamma_n$ of $P_{k_l}$.

We have described all diagrams giving a positive Artin presentation. Now we describe the pure $n$-braids determined by these presentations. Given a set of paths $\gamma_i$ in a disk with holes $W$, we proceed as in Section \ref{section2}. Consider $D\times I$, and the trivial $n$-braid $\beta_0$ on it. Suppose that $W$ is identified with $D\times \{ 1\}$. Take the union of the braid $\beta_0$ and the paths $\gamma_i$, which produce a braid whose endpoints lie, in one side, in the interior of $D \times \{ 0 \}$, and in the other side in $\partial_0 W$. To recover the desired braid, we just have to untwist the paths in $W$ as in Section \ref{section2}, but this can also be obtained by just pushing the arcs of $\beta$ into the interior of $D\times I$, keeping its endpoints fixed.

Suppose first that $n=2$. It follows from Figure \ref{2diagram} that the braid $\beta$ is obtained by performing full positive twists around the two chords of the braid and then twists on each component to adjust the framing. Suppose now that $n=3$. It follows from Figure \ref{3diagram} that the braid is obtained by performing full positive twists around the three chords of the braid, and then full positive twists around the first two chords, or well, full positive twist along the chords 2 and 3. The case of Figure \ref{3diagram} D is somewhat special. In this case the braid is obtained by performing at least one positive twist along the three chords of the braid and then a negative twist along the components 2 and 3 of the braid.

Consider now the general case of an $n$-artinian presentation. Take the braid $\beta$ as before, consisting of the trivial braid union paths in $W$. By pushing it into the interior of $D\times I$ in an orderly manner, by pushing first the outer parts, we see that we have full positive twists along chords in a block of chords. In the cases where a chord skips some of the holes before arriving at its respective hole, as in Figures \ref{ndiagram}, \ref{ndiagrams}, this can be obtained by first making a positive full twist in a block and then a negative full twist along a subblock of it. Continuing in this manner we get a description of the braid $\beta$ as an element of $\mathcal{P}$.
\end{proof}

Let $\beta \in \mathcal{P}$ be an unframed braid. By the construction of the previous result, it is cleat that $\beta$ can be represented by a positive Artin presentation. However, if $\beta$ is framed, the framing may introduce many negative twists around a hole, producing that a word in the Artin presentation terminates with a sequence of several $x_i^{-1}$.

 \begin{theorem}\label{artinstrongly} Let $\beta \in P_n$ be a  pure $n$-braid and let $\hat \beta$ be its closure. If $\beta $ represents a positive artinian  $n$-presentation, then $\hat \beta$ is strongly invertible. 
 \end{theorem}
\begin{proof}
This follows from Theorem \ref{main} and Lemma \ref{strongly}.
 \end {proof}

 \begin{theorem}\label{doublecover}
     Let $M$ be a closed 3-manifold whose fundamental group admits a positive artinian presentation, then $M$ is obtained by Dehn surgery on a strongly invertible link. In particular, $M$ is a double cover of $S^3$ branched along certain link. Furthermore, $M$ is a (possibly degenerate) graph manifold.
 \end{theorem}
 
 \begin{proof}
 This follows from Theorems \ref{main} \ref{artinstrongly}, by the Montesinos trick \cite{M}, and by Lemma \ref{surgery}
 \end{proof}

 Although we can say that the manifolds with positive Artin presentations are certain graph manifolds, we have not given a precise description of such manifolds.
 
\section{Examples}\label{examples}

In this section we give examples of integral framed, closed pure $n$-braids  $\hat \beta$, which produce positive Artin presentations.

\vskip 10pt

\begin{theorem} \label{3pure}
Let $\beta $ be a pure $3$-braid, and $\hat \beta $ the integral framed $(m_1, m_2, m_3)$ link  obtained by closing the braid.  Then $\hat \beta$  produces  a $3$-manifold which fundamental group admits a positive Artin presentation, if and only if one of the following conditions is satisfied:
\begin{enumerate}

\item 
$$\hat \beta = \widehat {\sigma_1^{2e_1}({\sigma_2\sigma_1\sigma_2})^{2e}} $$
  with  $m_1\geq e+e_1$, $m_2 \geq e+e_1\geq 0$, $m_3 \geq e$, $e\geq 0$, $e_1\geq 0$.
  \item
  $$\hat \beta = \widehat {\sigma_2^{2f_1}({\sigma_2\sigma_1\sigma_2})^{2e}} $$
  with  $m_1 \geq e$, $m_2 \geq e+f_1$, $m_3 \geq e+f_1$, $e\geq 0$, $f_1\geq 0$,
  or with $m_1\geq e$, $m_2\geq e-1$, $m_3 \geq e-1$, $e\geq 1$, $f_1=-1$.

  \end{enumerate}
  
\end{theorem}
 
 \begin{proof}
 By \cite{FN}, it follows that we have a diagram representing a pure $3$-braid (see Figure 3 in \cite{A}), in particular in Figure \ref{3braidsmall} we have such a diagram for the small case where we have only 3 boxes, $e, e_1, f_1$, which means that $\beta = {\sigma_1^{2e_1}\sigma_2^{2f_1}({\sigma_2\sigma_1\sigma_2})^{2e}} $. Observe that in the small case, all closed, pure $3$-braids are strongly invertible links. From \cite {A}, we know that an Artin presentation of the fundamental group of a $3$-manifold obtained by integral Dehn surgery on $\hat \beta$, with surgery coefficients $m_1,m_2,m_3$ is given by the generators $x_1, x_2, x_3$ and relations (this is in fact a cyclic permutation of the presentation given in \cite{A}):
 
 $r_1 = (x_1x_2x_3)^{e}  (x_1 (x_2x_3)^{f_1}x_2(x_2x_3)^{-f_1})^{e_1}x_1^{m_1-e-e_1} $

$ r_2 = (x_1x_2x_3)^{e}(x_2x_3)^{f_1} (x_1 (x_2x_3)^{f_1}
x_2(x_2x_3)^{-f_1})^{e_1}x_2^{m_2-e-e_1-f_1}$

$r_3 =(x_1x_2x_3 )^e  (x_2x_3 )^{f_1}x_3^{m_3-e-f_1}$ 

By Theorem \ref{main}, for this presentation to be positive is required that the 3-braid is small, and furthermore that $e_1=0$ or $f_1=0$. 

Suppose first that $f_1=0$. In this case the presentation is given by 

$r_1 = (x_1x_2x_3)^{e}  (x_1x_2)^{e_1}x_1^{m_1-e-e_1} $

$ r_2 = (x_1x_2x_3)^{e}(x_1x_2^{e_1}x_2^{m_2-e-e_1}$

$r_3 =(x_1x_2x_3 )^e x_3^{m_3-e}$ 

Also, $e_1$ cannot be negative, so $e_1\geq 0$, and then for the presentation to be positive is required that $m_1\geq e+e_1$, $m_2 \geq e+e_1\geq 0$, $m_3 \geq e$, $e\geq 0$, $e_1\geq 0$.

Suppose now that $e_1=0$. In this case the presentation is given by

$r_1 = (x_1x_2x_3)^{e}x_1^{m_1-e} $

$ r_2 = (x_1x_2x_3)^{e}(x_2x_3)^{f_1}x_2^{m_2-e-f_1}$

$r_3 =(x_1x_2x_3 )^e  (x_2x_3 )^{f_1}x_3^{m_3-e-f_1}$ 

If $f_1\geq 0$, then for the presentation to be positive is required that $m_1 \geq e$, $m_2 \geq e+f_1$, $m_3 \geq e+f_1$, $e\geq 0$, $e_1\geq 0$. 

If $f_1=-1$, then for the presentation to be positive is required that $m_1\geq e$, $m_2\geq e-1$, $m_3 \geq e-1$, $e\geq 1$.
\end{proof}

\begin{figure}[h]
\centering
\includegraphics[width=3cm]{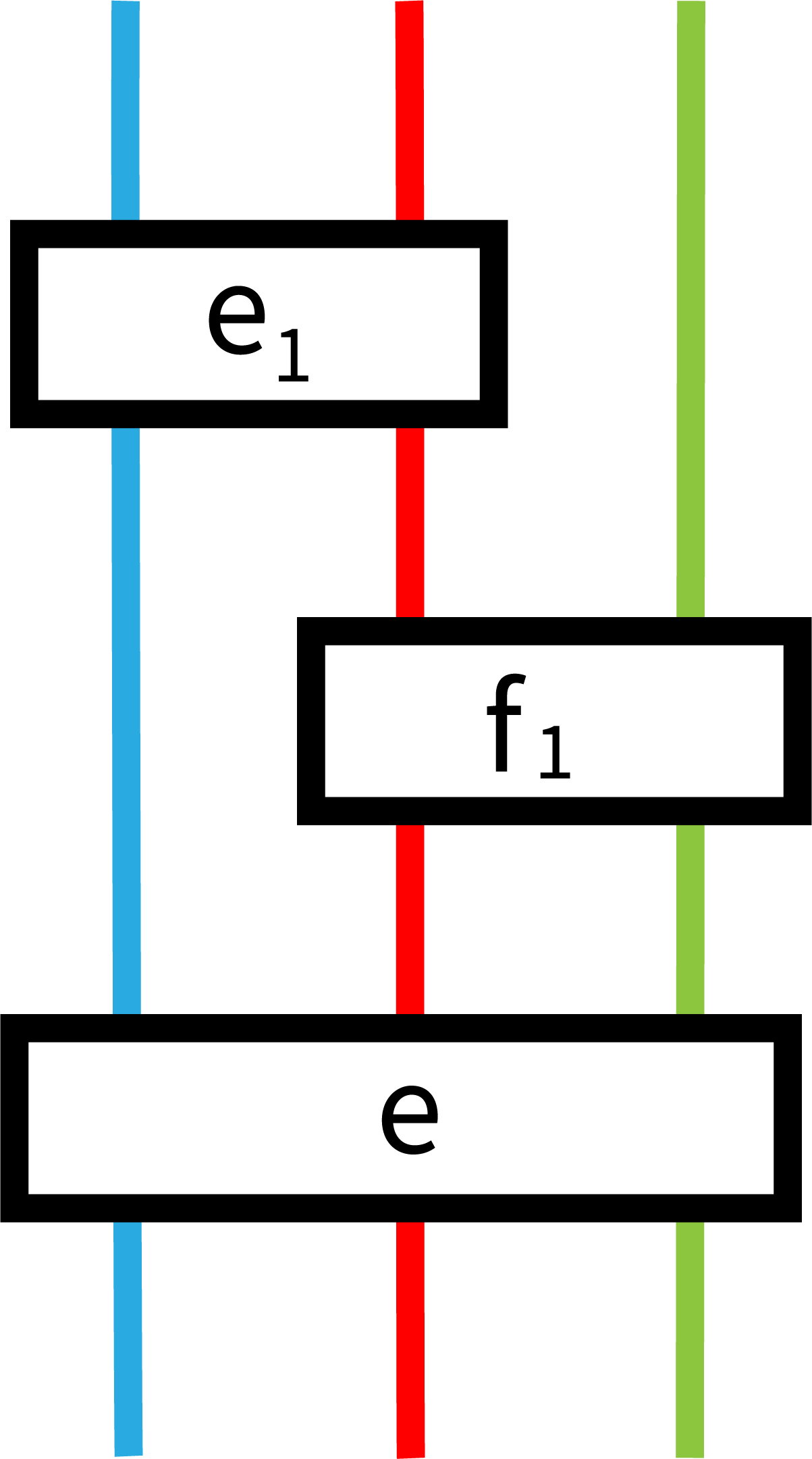} 
\caption{Small type.}
\label{3braidsmall}
\end{figure}

 It can be seen from Theorem 3.3, that there are  strongly invertible links $\hat \beta$ which do not produce a positive Artin presentation. In Figure \ref{nonpositives} we give some examples of nonpositive Artin presentations. These come from small pure 3-braids, hence strongly invertible links.

\begin{figure}[h]
\centering
\includegraphics[width=10cm]{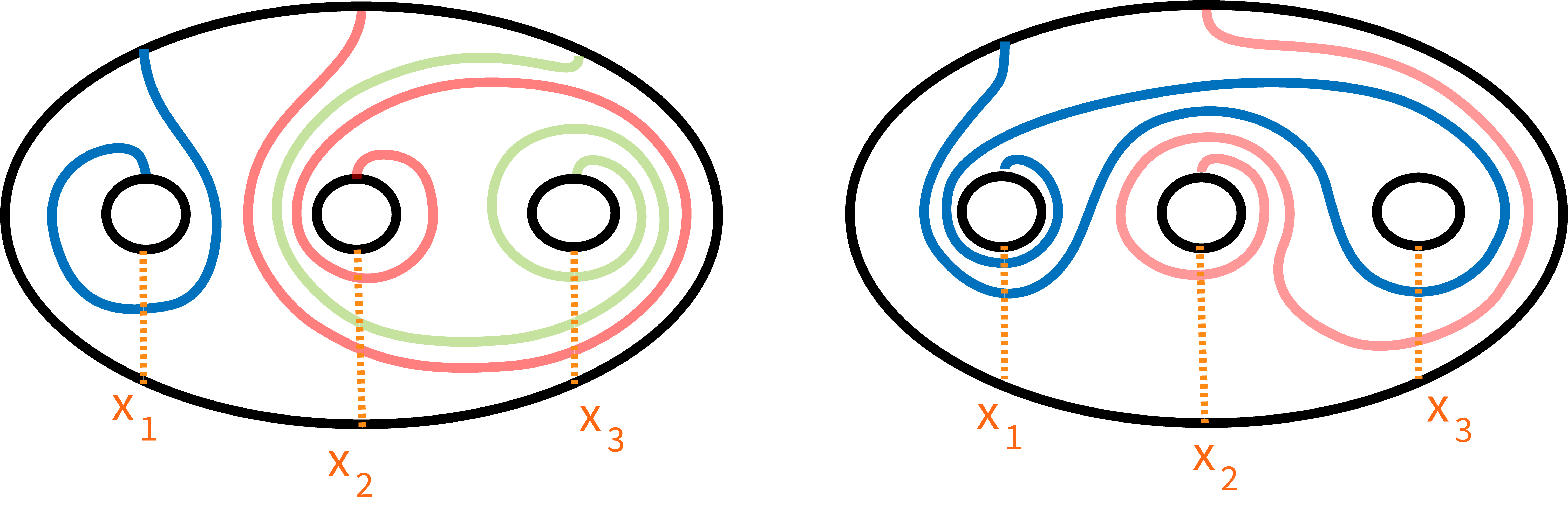} 
\caption{Examples of nonpositive Artin presentations.}
\label{nonpositives}
\end{figure}

\vskip 20pt

Now, we give some calculations before stating  the next proposition. In this part, we  will calculate the fundamental group of a 3-manifold obtained from an integral framed link $\hat \beta$ of $n$ chords, such that $\beta \in \mathcal{P}$. 

 Before continuing, let us  remember the operation  $''\circ''$ defined in \cite {A}.  
Let
$r, s$ be two $n$-Artin presentations, where $r= ( r_1,r_2,\cdots, r_n)$ and  $s= ( s_1, s_2,\cdots , s_n)$, 
and $s\circ r = s\cdot r^{s}$, 
with $r^{s} = (r_1^{s},\cdots, r_n^{s})$, where $r_i^s$ means that we are taking the conjugate of each $x_j$ appearing in the relation $r_i$ by $s_j$, i.e. $s_j^{-1}x_j s_j$ and then we consider the product $\prod s_j^{-1}x_js_j$.  Also $s\cdot r^s = (s_1r_1^s, s_2r_2^s,\cdots , s_nr_n^s)$. 

Now we see an example to illustrate the operation.

Consider the Artin presentations given by $r_1 = x_1x_2x_3x_1^2, r_2 = x_1x_2x_3x_2x_3x_2, r_3  =  x_1x_2x_3 x_2x_3^2$, such that  $r= (r_1, r_2, r_3)$ and $s_1 = x_1x_2x_3x_1, s_2 = x_1x_2x_3x_2^2, s_3 =  x_1x_2x_3^2$, so $s= (s_1, s_2, s_3)$. Then  we calculate $r^s$:

\noindent $ r_1^s = (x_1x_2x_3x_1x_1 ) ^s = (s_1^{-1}x_1s_1)(s_2^{-1}x_2s_2)(s_3^{-1}x_3s_3)(s_1^{-1}x_1s_1)(s_1^{-1}x_1s_1)$

\noindent $ r_2^s = (x_1x_2x_3x_2x_3x_2)^s = (s_1^{-1}x_1s_1)(s_2^{-1}x_2s_2)(s_3^{-1}x_3s_3)(s_2^{-1}x_2s_2)(s_3^{-1}x_3s_3)(s_2^{-1}x_2s_2)$

\noindent $ r_3^s = (x_1x_2x_3x_2x_3x_3)^s = ( s_1^{-1}x_1s_1)(s_2^{-1}x_2s_2)( s_3^{-1}x_3s_3)( s_2^{-1}x_2s_2)(s_3^{-1}x_3s_3)(s_3^{-1}x_3s_3)$

\noindent So,
  $$s\cdot r^{s} = ( x_1x_2x_3x_1 r_1^s, x_1x_2x_3x_2^2 r_2^s, x_1x_2x_3^2 r_3^s)$$
from \cite {A} it follows  that this is an artinian $n$-presentation.

\vskip 10pt

Now we consider the next artinian $n$-presentations given for an element of $\mathcal{P}$, where $n=k_1+k_2+k_3$.
Let  $ \Delta _{k_1}^{e_1}$  and $ \Delta _{k_2}^{e_2}$ be two Artin presentations,  given by:

\noindent $$ \Delta _{k_1}^{e_1} = \big ( ( x_1x_2x_3\cdots  x_{k_1})^{e_1},  \cdots , ( x_1x_2x_3\cdots  x_{k_1})^{e_1}, 1, 1, 1, \cdots ,1\big)$$

\noindent  i.e. the first $k_1$ entries are equal to $( x_1x_2x_3\cdots  x_{k_1})^{e_1}$ and the last $n-k_1$ entries are the identity,

\noindent $$ \Delta _{k_2}^{e_2} = \big ( 1, \cdots, 1, ( x_{k_1+1} x_{k_1+2}\cdots x_{k_1+k_2})^{e_2},  \cdots , ( x_{k_1+1} x_{k_1+2} \cdots  x_{k_1+k_2})^{e_2},  1,\cdots ,1\big)$$

\noindent where the first $k_1$ entries are $1$,  the following $k_2$ entries are $( x_{k_1+1} x_{k_1+2}\cdots x_{k_1+k_2})^{e_2}$  and finally  the last $n-k_1-k_2$  terms are the identity.
 Then, we calculate $ \Delta _{k_2}^{e_2} \circ  \Delta _{k_1}^{e_1}$.  Observe that we conjugate the first $k_1$ terms  by the identity,  and  get: 

$$ \big ( \Delta _{k_1}^{e_1} \big ) ^{\Delta _{k_2}^{e_2}} = \big ( ( x_1x_2x_3,\cdots x_{k_1})^{e_1},  \cdots , ( x_1x_2x_3\cdots  x_{k_1})^{e_1}, 1, 1, 1, \cdots ,1\big)$$

\noindent So, the same artinian $n$-presentation $\Delta _{k_1}^{e_1}$ is obtained. Now we have just to calculate the product entry by entry of these  two Artin presentations to get:

\noindent $$\big ( ( x_1x_2x_3\cdots  x_{k_1})^{e_1},  \cdots , ( x_1x_2x_3\cdots  x_{k_1})^{e_1},  ( x_{k_1+1}x_{k_1+2}\cdots  x_{k_1+k_2})^{e_2},  \cdots,$$ $$\cdots , 
( x_{k_1+1}x_{k_1+2}\cdots x_{k_1+k_2})^{e_2}, 1, \cdots ,1\big)$$

\noindent where the first $k_1$ entries are equal to $( x_1x_2x_3\cdots  x_{k_1})^{e_1}$, 
the following $k_2$ entries are $( x_{k_1+1} x_{k_1+2}\cdots x_{k_1+k_2})^{e_2}$  and finally  the last $n-k_1-k_2$  terms are the identity.

Now we calculate the Artin presentation for some specific braids in $\mathcal{P}$. Suppose that $\beta \in \mathcal{P}$ is an element of $n$ chords, where we have the partition $P_{k_1}, P_{k_2}, P_{k_3}$ and $P_{k_4}$ such that $k_1+k_2+k_3+k_4 =n$. And let $P_{k_2}\cup P_{k_3}$ and  $ P_{k_1} \cup P_{k_2} \cup P_{k_3} \cup P_{k_4}$  be two blocks, one completely contained in the other. Suppose that all of the full twists performed, $e_{k_1+k_2+k_3+k_4}$, $e_{k_2+k_3}$, $e_{k_1}$, $e_{k_2}$, $e_{k_3}$, $e_{k_4}$, are positive. We calculate the relations.

 \begin{proposition}\label{calculo1} Let  $\hat \beta$ be an integral framed closed pure $n$-braid  with   frame\break   $(m_1, m_2,\cdots , m_n)$, such that $\beta \in \mathcal{P}$. Suppose that $\beta$ has a decomposition as described above.
 Then the associated Artin presentation has  generators $x_1, x_2,\cdots , x_n$ and  relations:

 \noindent 1) for $1\leq j\leq  k_1$
 
 $$r_j = (x_1x_2\cdots x_n)^{e_{\sum_{i=1}^4 k_i}}(x_1x_2x_3\cdots x_{k_1})^{e_{k_1}}x_j^{m_j-e_{k_1} -e_{(\sum_{i=1}^4k_i)}}$$

\noindent  2) for $k_1+1 \leq j \leq k_1+k_2$

$$r_{j} = (x_1x_2\cdots x_n)^{e_{(\sum_{i=1}^4 k_i)}}(x_{k_1+1}x_{k_1+2}x_{k_1+3}.
\cdots x_{k_1+k_3})^{e_{k_2+k_3}}\cdots$$

 $$\cdots(x_{k_1+1}x_{k_1+2}x_{k_1+3}\cdots x_{k_1+k_2})^{e_{k_2}}x_{j}^{m_{j}-e_{k_{2}}-(e_{k_2+k_3})-e_{(\sum_{i=1}^4 k_i)}}
$$

 \noindent  3) for $k_1+k_2+1 \leq j \leq k_1+k_2 +k_3$

$$ r_{j} = (x_1x_2\cdots x_n)
 ^{e_{({\sum_{i=1}^4k_i})}}(x_{k_1+1} x_{k_1+2}x_{k_1+3}\cdots x_{k_1+ k_3})^{e_{k_2+k_3}}\cdots$$
 
 $$\cdots(x_{k_1+k_2+1}x_{k_1+k_2+2}x_{k_1+k_2+3}\cdots x_{k_1+k_2+k_3})^{e_{k_3}}x_{j}^{m_{j} - e_{k_{3}}-(e_{k_2+k_3})-e_{(\sum_{i=1}^4 k_i)}}$$

 \noindent 4) for $k_1+k_2+k_3+1\leq j\leq  k_1+k_2+k_3+k_4$

$$r_j = ( x_1x_2\cdots x_n)^{e_{(\sum_{i=1}^4k_i)}} (x_1x_2x_3\cdots x_{k_1})^{e_{k_4}}x_j^{m_j-e_{k_4} -e_{(\sum _{i=1}^4k_i)}}$$
 
  \end{proposition}

 \begin{proof} It follows from the previous calculations and the observation that the order in which the blocks are twisted can be interchanged.  \end{proof}
 
 It is easy to see that in an arbitrary element $\beta \in \mathcal{P}$,  in each relation we added only the  product corresponding to one block where the chord appears.
 
\vskip10pt

 \vskip 10pt
 
 \begin{theorem}\label{calculo2}
Let  $\beta \in \mathcal{P}$ such that all twists performed are positive. Let $\hat \beta $ the   integral framed link with frame $(m_1, m_2,\cdots , m_n)$  obtained by closing the braid. Then  $\hat \beta$  admits a positive Artin $n$-presentation if and only if  for all  $j$, 
 $$ m_j -  \big (\sum{ e_{k_j} }\big ) \geq 0 $$
\noindent where the sum runs over all the blocks formed by unions of the $P_{k_j}$'s  containing the $j$-th chord.

\end{theorem}
 
 \begin{proof}
 An artinian presentation can be calculated just as in Proposition \ref{calculo1}. Each relation will be positive, except at the end, where the relation terminates in a word $x_j^{m_j -  \big (\sum{ e_{k_j} }\big ) }$.
 \end{proof}

\vskip20pt

According to Theorem \ref{3pure}, for a closed pure 3-braid to produce a positive Artin presentations, the braid must belong to a very restrictive class. Furthermore, there are 3-braids in the small case, that do not produce a positive Artin presentation, but such that the closed braid is strongly invertible. In this examples is possible that the given manifolds have some other presentation as surgery on a closed braid that produce a positive Artin presentation.

On the other hand
  the next example shows  that there are $3$-manifolds whose fundamental group does not admit a positive artinian $n$-presentation.
  
 This example is the $3$-torus $S^1\times S^1\times S^1$. It is possible to say that it does not have a positive 
Artin presentation since this $3$-manifold is not a double branched cover of the $3$-sphere \cite{M1}, it means that it is not obtained by surgery on a strongly invertible link.
 
 An artinian $n$-presentation of the 3-torus is the following, which was given by  F. Gonz\'alez-Acu\~na, see Figure \ref{3toro}.
 
 $$\pi_1(M) = < x_1, x_2, x_3 \, \vert \, x_1x_3x_1^{-1}x_3^{-1}x_2x_3x_1x_3^{-1}x_1^{-1}x_3^{-1}x_2^{-1}x_3, \, x_1x_3x_1^{-1}x_3^{-1}, \, x_3^{-1}x_2^{-1}x_3x_1^{-1}x_3^{-1}x_2x_3x_1>$$

\begin{figure}
\centering
\includegraphics[width=8cm]{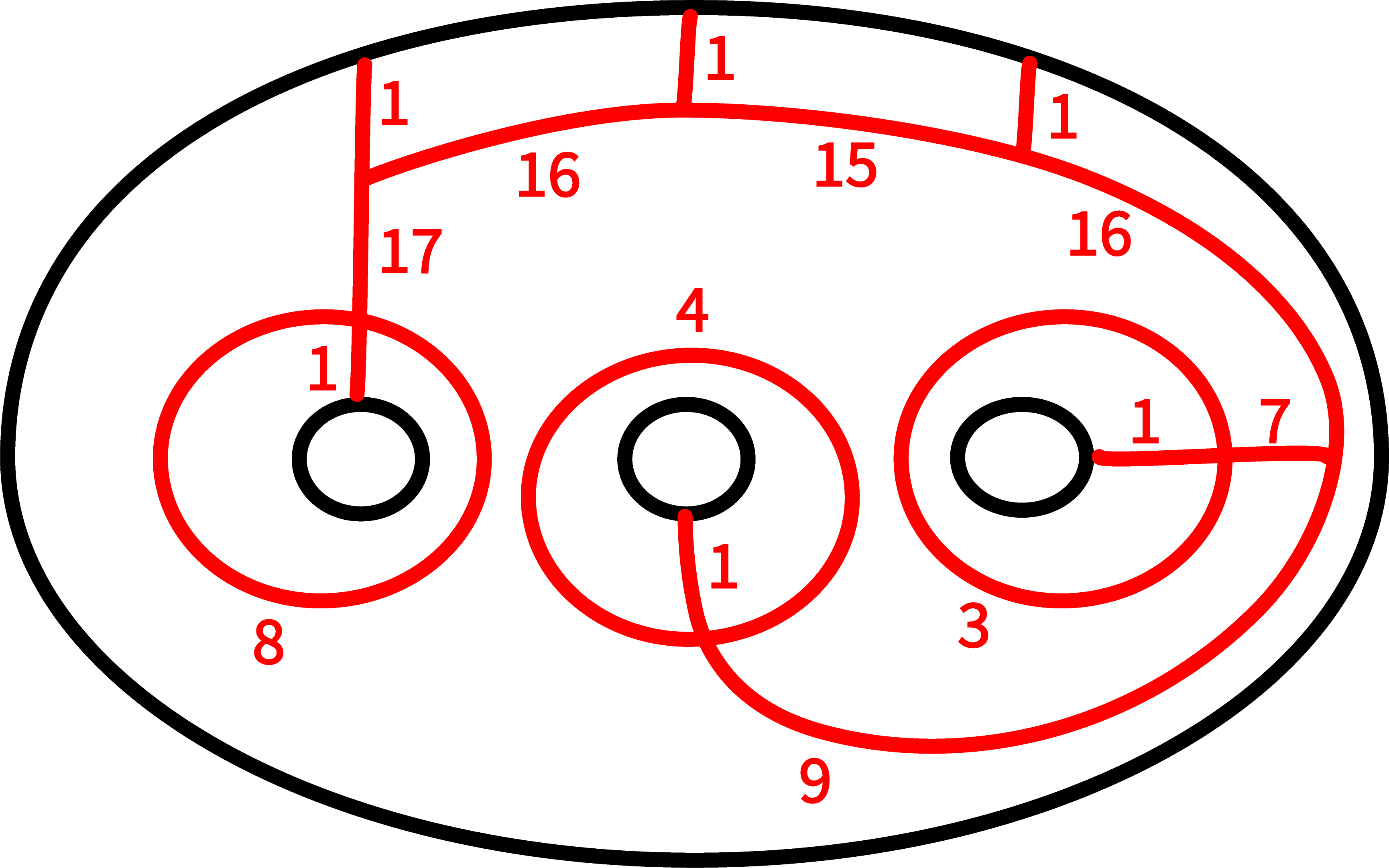} 
\caption{3-torus.}
\label{3toro}
\end{figure}

Here, the diagram of the disk with holes and paths has been simplified. Each arc represents a set of arcs as indicated by the numbers.

\vskip10pt

\bibliographystyle{amsplain}

\end{document}